\documentclass[12pt,a4paper,oneside]{amsart}

\usepackage{amsfonts, amsmath, amssymb, amsthm, amscd}
\usepackage{hyperref}
\hypersetup{
  colorlinks   = true, 
  urlcolor     = blue, 
  linkcolor    = blue, 
  citecolor   = red 
}
\usepackage{anysize}
\marginsize{2cm}{2cm}{1.5cm}{1.5cm}
\usepackage[pdftex]{graphicx}

\newtheorem{theorem}{Theorem}[section]

\newtheorem{proposition}[theorem]{Proposition}

\theoremstyle{definition}

\theoremstyle{remark}
\newtheorem{remark}[theorem]{Remark}
\newtheorem{question}[theorem]{Question}
\newtheorem{example}[theorem]{Example}

\newcounter{fig}
\newcommand{\f}{\refstepcounter{fig} Fig. \arabic{fig}. }

\numberwithin{equation}{section}

\DeclareMathOperator{\vol}{vol}

\DeclareMathOperator{\supp}{supp}
\DeclareMathOperator{\Hess}{Hess}
\DeclareMathOperator{\tr}{tr}

\renewcommand{\epsilon}{\varepsilon}

\renewcommand{\phi}{\varphi}

\title[Gromov's waist of non-radial Gaussian measures\dots]{Gromov's waist of non-radial Gaussian measures and radial non-Gaussian measures}

\author{Arseniy~Akopyan{$^\spadesuit$}}

\email{akopjan@gmail.com}

\author{Roman~Karasev{$^\clubsuit$}}

\email{r\_n\_karasev@mail.ru}
\urladdr{http://www.rkarasev.ru/en/}

\thanks{{$^\spadesuit$} Supported by the European Research Council (ERC) under the European Union's Horizon 2020 research and innovation programme (grant agreement  No 716117)}
\thanks{{$^\clubsuit$} Supported by the Federal professorship program grant 1.456.2016/1.4 and the Russian Foundation for Basic Research grants 18-01-00036 and 19-01-00169}

\address{{$^\spadesuit$} Institute of Science and Technology Austria (IST Austria), Am Campus 1, 3400 Klosterneuburg, Austria}
\address{{$^\clubsuit$} Moscow Institute of Physics and Technology, Institutskiy per. 9, Dolgoprudny, Russia 141700}
\address{{$^\clubsuit$} Institute for Information Transmission Problems RAS, Bolshoy Karetny per. 19, Moscow, Russia 127994}

\begin{document}

\begin{abstract}
We study the Gromov waist in the sense of $t$-neighborhoods for measures in the Euclidean space, motivated by the famous theorem of Gromov about the waist of radially symmetric Gaussian measures. In particular, it turns our possible to extend Gromov's original result to the case of not necessarily radially symmetric Gaussian measure. We also provide examples of measures having no $t$-neighborhood waist property, including a rather wide class of compactly supported radially symmetric measures and their maps into the Euclidean space of dimension at least 2. 

We use a simpler form of Gromov's pancake argument to produce some estimates of $t$-neighborhoods of (weighted) volume-critical submanifolds in the spirit of the waist theorems, including neighborhoods of algebraic manifolds in the complex projective space.

For reader's convenience, in one appendix of this paper we provide a more detailed explanation of the Caffarelli theorem that we use to handle not necessarily radially symmetric Gaussian measures. In the other appendix, we provide a comparison of different variations of Gromov's pancake method.
\end{abstract}

\subjclass[2010]{49Q20,53C20,53C23}

\maketitle

\section{Introduction}

In \cite{grom2003} Mikhail Gromov proved the waist of the Gaussian measure theorem: For any continuous map $f : \mathbb R^n\to\mathbb R^k$ and a radially symmetric Gaussian measure $\gamma$, it is possible to find a fiber $f^{-1}(y)$ such that for any $t>0$
\[
\gamma(f^{-1}(y) + t) \ge \gamma(\mathbb R^{n-k} + t),
\]
where $X+t$ denotes the $t$-neighborhood of a set $X$ and $\mathbb R^{n-k}\subset \mathbb R^n$ denotes any $(n-k)$-dimensional linear subspace of $\mathbb R^n$. This statement is a very general version of the concentration phenomenon that was previously extensively studied for functions, the case $k=1$ of the result.

In the recent papers \cite{klartag2016,ak2016ball,akopyan2017cat} a Minkowski content version of such a result was considered, where one does not require the inequality for all $t$, but requires a lower bound on the asymptotics for $t\to +0$. In this simplified problem it turned out that $\gamma$ can be replaced by other measures, such as the uniform measure on a centrally symmetric cube \cite{klartag2016}, the uniform measure on a Euclidean ball \cite{ak2016ball}, the uniform measure on a ball in the spaces of constant curvature \cite{akopyan2017cat}, and some other examples.

In this paper we address the question of extending the precise $t$-neighborhood waist theorem to measures other than the radially symmetric Gaussian measures, in the general understanding of \cite[Question 3.1.A]{grom2003}. More precisely, for a given finite Borel measure $\mu$ in $\mathbb R^n$ we ask if the following is true: For any continuous $f : \mathbb R^n\to \mathbb R^k$ there exists a fiber $f^{-1}(y)$ such that for any $t>0$
\[
\mu(f^{-1}(y) + t) \ge \mu(\mathbb R^{n-k} + t).
\]
Even restricting ourselves to radially symmetric measures, we find that the situation is not easy. A counterexample was given in \cite[Remark 5.7]{akopyan2017cat} showing that for $\mu$ uniform in the Euclidean ball of sufficiently high dimension there is no such precise $t$-neighborhood waist theorem.

In Section \ref{section:radial-non-gaussian} of this paper we give other counterexamples and discuss some positive results in the plane. In particular, Theorem \ref{theorem:counterexample-k-ge-2} asserts that no such $t$-neighborhood waist theorem is possible for compactly supported radial measures and $2\le k < n$.

On the positive side, we have the following extension of Gromov's theorem to the case of not necessarily radially symmetric Gaussian measures:

\begin{theorem}
\label{theorem:non-radial-gaussian}
Let $a_1\ge a_2 \ge \dots \ge a_n > 0$ and let $\gamma$ be a Gaussian measure in $\mathbb R^n$ with density
\[
\rho = e^{-a_1x_1^2 - a_2x_2^2 - \dots - a_nx_n^2}.
\]
Then for any continuous map $f : \mathbb R^n\to \mathbb R^k$ there exists a point $y\in\mathbb R^k$ such that for any $t>0$
\[
\gamma(f^{-1}(y) + t) \ge \gamma(\mathbb R^{n-k} + t).
\]
It is important that by $\mathbb R^{n-k}\subset\mathbb R^n$ in this theorem we denote the coordinate subspace spanned by the first $n-k$ coordinates.
\end{theorem}

The case of an arbitrary Gaussian measure with full-dimensional support is reduced to this one by translation and diagonalization of its covariance matrix, one only needs to care that a particular $(n-k)$-dimensional subspace is assumed in the right-hand side of the inequality.

To prove this theorem we use Gromov's original scheme, replacing the hard argument about choosing ``centers'' of the pancakes with an application of the Caffarelli theorem on $1$-Lipschitz monotone transportation. This not only allows to extend the result to the non-radial case, but also simplifies the original argument, at least in our understanding.

After exhibiting, in Section \ref{section:radial-non-gaussian}, examples of radially symmetric measures, for which the neighborhood waist theorem does not hold, we address the question of choosing $y=0$ in the waist theorems for odd maps in Section \ref{section:odd}. It turns out that the neighborhood waist theorem holds true for all radially symmetric measures and odd maps, as Theorem \ref{theorem:neighborhood-odd} asserts.

In Section \ref{section:manifold-neighborhood} we naturally pass to the following question: For which measure $\mu$ and a submanifold $X\subset \mathbb R^n$, or $X\subset\mathbb S^n$, can we claim the same neighborhood estimate as in the waist theorem? This question is inspired by the result of \cite{klartag2017}, and in particular we observe that the right neighborhood estimate holds for volume-critical submanifolds of $\mathbb S^n$ with the uniform measure, see Proposition \ref{proposition:sphere-critical}, and that the neighborhood waist theorem turns out to be true for all radially symmetric measures and homogeneous holomorphic maps, see Theorem \ref{theorem:radial-symmetric-homo-holo}.

In Appendix \ref{section:caffarelli} we give the statement and an explanation of the Caffarelli theorem, in Appendix~\ref{section:pancakes} we discuss Gromov's method of producing pancakes in Gromov's waist theorem.

\subsection*{Acknowledgments} The authors thank Alexey Balitskiy, Michael Blank, Alexander Esterov, Sergei Ivanov, Bo'az Klartag, Jan Maas, and the unknown referee for useful discussions, suggestions, and questions.

\section{Nonradial Gaussian measures: Proof of Theorem \ref{theorem:non-radial-gaussian}}	
\label{section:non-radial-gaussian}

Let us start with establishing the linear case of the theorem. To do so, it is sufficient to consider the orthogonal projection $f : \mathbb R^n\to A$ for a linear subspace of dimension at most $k$, and for $B=A^\perp$ we need to find the lower bound on $\gamma(\nu_t(B))$. Here we use more explicit notation $\nu_t(X)$ for the $t$-neighborhood of a point or of a set $X$. We can also assume $\dim A$ to be precisely $k$, since further projections onto subspaces of smaller dimension may only improve the required estimate. 

Note that the covariance form of the pushforward $f_* \gamma$ is obtained by restricting the covariance form of $\gamma$ to $A$. It is a folklore fact that the restriction of a quadratic form with eigenvalues $\lambda_1\le\lambda_2\le\dots\le\lambda_n$ produces a quadratic form with eigenvalues $\kappa_1\le\ldots\le\kappa_k$ such that 
\[
\kappa_1 \ge \lambda_1, \ldots \kappa_k \ge \lambda_k\quad\text{and}\quad \kappa_1\le \lambda_{n-k+1},\dots, \kappa_k\le\lambda_n.
\]
This fact is easily proven to induction, restricting the quadratic form to a codimension $1$ subspace several times. Hence if we take the orthonormal coordinates $y_1,\ldots,y_k$ in $A$ then the density of $f_* \gamma$ up to multiplication by a constant will be $e^{-b_1y_1^2 - \dots - b_k y_k^2}$ with 
\[
b_1 \ge a_{n-k+1},\dots, b_k \ge a_n
\]
and therefore
\begin{equation}
\label{equation:linearbound}
\frac{\gamma(\nu_t(B))}{\gamma(\mathbb R^n)} \ge \frac{\int_{\nu_t(0)} e^{-b_1y_1^2 - \dots - b_k y_k^2}\; dy_1\dots dy_k}{\int_{\mathbb R^k} e^{-b_1y_1^2 - \dots - b_k y_k^2}\; dy_1\dots dy_k} \ge \frac{\int_{\nu_t(0)} e^{-a_{n-k+1}x_{n-k+1}^2 - \dots - a_n x_n^2}\; dx_{n-k+1}\dots dx_n}{\int_{\mathbb R^k} e^{-a_{n-k+1}x_{n-k+1}^2 - \dots - a_n x_n^2}\; dx_{n-k+1}\dots dx_n},
\end{equation}
the last inequality is done by the substitutions $t_i^2 = b_ix_i^2$ and $t_i^2 = a_{n-k+i}x_{n-k+i}^2$, after these substitutions the same expression in the two sides of the inequality is integrated over the two domains, one containing the other. Thus the theorem holds for linear maps with a simple choice $y=0$.

For the general case, we follow Gromov's argument for the round Gaussian measure from \cite{grom2003} (see also the explanations in \cite{klartag2016}). Let us state an appropriate version of the pancake decomposition theorem:

\begin{theorem}[Essentially due to Gromov, \cite{grom2003}]
\label{theroem:pancakes}
Let $\mu$ be a measure in the ball $B(R)\subset\mathbb R^n$ with density bounded from below and from above by positive numbers. Let $\mathcal K(B(R))$ be all convex bodies in $B(R)$ with nonempty interior and assume we have a map $F : \mathcal K(B(R)) \to \mathbb R^k$ continuous in the Hausdorff metric, for some $1\le k < n$.

Then for any power of two $N = 2^I$ it is possible to produce a binary decomposition of $B(R)$ into convex parts $P_1,\ldots,P_N$ so that:

1) $\mu(P_1) = \dots = \mu(P_N)$;

2) $F(P_1) = \dots = F(P_N)$;

3) Given $\delta > 0$, for sufficiently large $N$ depending on $\mu$ and $\delta$, all the parts $P_i$ will be $\delta$-close to $k$-dimensional affine subspaces.
\end{theorem}

The last property is called the \emph{$k$-pancake property}. The measure we are going to plug into this theorem is the restriction of $\gamma$ to a big ball $B(R)$. The map $F$ will be composed of $f$ given in the waist theorem and a certain selection of a ``center'' $c(P_i)$ of a part with non-empty interior, which is the crucial part of our argument, different from the selection in \cite{grom2003,klartag2016} and other previous works. Since the complement of $B(R)$ has arbitrarily small Gaussian measure, this will result in an arbitrary small error term in the estimate, which is handled by a compactness argument for possible values of $y$, the common value of $f(c(P_i))$. We give more explanations on the pancake decomposition in Appendix~\ref{section:pancakes}, outlining a version of the proof of Theorem \ref{theroem:pancakes}.

The selection of the centers $c(P_i)$ of $P_i$ is the crucial part of the argument, we need to do it so that the ratio
\begin{equation}
\label{equation:ratio}
\frac{\gamma( \nu_t(c(P_i)) \cap P_i )}{\gamma( P_i )}
\end{equation}
is bounded from below by a certain constant so that the summation of such lower bounds produces the required total estimate. 

Consider the restriction $\mu_i = \gamma|_{P_i}$, we are interested in the estimate of \eqref{equation:ratio} from below that now assumes the form
\begin{equation}
\label{equation:ratio2}
\frac{\mu_i( \nu_t(c(P_i)))}{\mu_i(\mathbb R^n)}.
\end{equation}

Let us make the monotone transportation $T_i$ of (an appropriately scaled) $\gamma$ to $\mu_i$. The measure $\mu_i$ is \emph{more log-concave} than $\gamma$, that is $\mu_i$ can be expressed as the product of $\gamma$ and a log-concave function, the characteristic function of $P_i$ in our case. In this situation the Caffarelli theorem \cite{caffarelli2000} (see also \cite{koles2011} and explanations in Section \ref{section:caffarelli}) implies that the monotone transportation $T_i$ is $1$-Lipschitz.  Let us define the center $c(P_i)= T_i(0)$. By the stability propery of the monotone transportation (see, for example, \cite[Proposition 3.2]{brenier1991}) the centers depend continuously on $P_i$ in the Hausdorff metric while $P_i$ keeps nonempty interior; thus we can plug this selection of the center into the Borsuk--Ulam type theorem for the configuration space of such partition (more details are in \cite{grom2003} and \cite{klartag2016}) to ensure that
\[
f(c(P_1)) = \dots = f(c(P_N)) = y_N.
\] 

In order to have the total estimate it remains to bound \eqref{equation:ratio} (or \eqref{equation:ratio2}) for a pancake $P_i$ which is $\delta$-close to an affine subspace $A_i$. For normalization, we translate $P_i$ and $\mu_i$ so that the center $c(P_i)$ becomes the origin. Denote the translated $\mu_i$ by $\kappa_i$. The monotone transportation of (an appropriately scaled) $\gamma$ into $\kappa_i$ is the composition of the original transportation $T_i$ of $\gamma$ to $\mu_i$ and the translation, this can be seen from the fact that if we add a linear function to the the potential of the monotone transportation then the transportation gets composed with a translation, below we work with potentials in more details. $1$-Lipschitz property is also retained under a translation. Eventually we consider a measure $\kappa_i$ obtained from (an appropriately scaled) $\gamma$ by a $1$-Lipschitz monotone transportation $T_i$ (denoted by the same letter as before the translation) taking origin to the origin, from the pancake property the support of $\kappa_i$ is $\delta$-close to a linear subspace $A_i$ (we keep the same letter for the translated $A_i$), and we need to have a lower bound for 
\begin{equation}
\label{equation:ratio3}
\frac{\kappa_i( \nu_t(0))}{\kappa_i(\mathbb R^n)}.
\end{equation}

Put for brevity $A = A_i$ and $B = A^\perp$, now they are linear subspaces, also denote by $\pi_A$ and $\pi_B$ the corresponding orthogonal projections. The map $T_i$ is a monotone transportation, by definition it means that it has a convex potential 
\[
U : \mathbb R^n \to \mathbb R,\quad T_i(x) = \nabla U (x),
\]
which must have the minimum at the origin, since we assume $T_i(0) = 0$. The convexity of $U$ means that the Hessian is positive semidefinite, and the $1$-Lipschitz property of $T_i$ means that the difference
\[
\Hess U - dx_1^2 - \dots - dx_n^2
\] 
is negative semidefinite. 
Here we use the fact that $\Hess U$ exists almost everywhere.

Now we are going to approximate $U$ with another potential $V$ satisfying the same property of convexity and the same bound on the Hessian, producing a monotone $1$-Lipschitz transportation $S : \mathbb R^n\to A$, showing that $T_i$ is also approximated by $S=\nabla V$. 

Choose $\epsilon > 0$ and then choose a radius $R$ so that 
\[
\gamma(\nu_R(0)) \ge (1 - \epsilon) \gamma(\mathbb R^n).
\]
Let $X=\nu_R(A)$, then we also have
\[
\gamma(X) \ge (1 - \epsilon) \gamma(\mathbb R^n).
\]
The assumption that the image of $T_i$ lies in the neighborhood $\nu_\delta(A)$ means that 
\begin{equation}
\label{equation:derivative}
|\pi_B (\nabla U(x))| \le \delta.
\end{equation}

If we put $V(x) = U(\pi_A (x))$ then $V$ is also convex, differentiable, and the inequality
\[
\Hess V - dx_1^2 - \dots - dx_n^2\le 0
\]
holds almost everywhere (here we again use the fact about the eigenvalues of the restriction of a quadratic form). Its gradient $S=\nabla V$ is $1$-Lipschitz, and by integration \eqref{equation:derivative} from $\pi_A (x)$ to $x$ we have
\[
|V(x) - U(x)| \le (R+\epsilon)\delta
\]
on the neighborhood $\nu_\epsilon(X)$. By our choice of coordinates both $U$ and $V$ have the minimum at the origin.

It is easy to see, that the estimate $|V - U| < 1/2\epsilon^2$ on $\nu_\epsilon(X)$ implies
\[
|\nabla V - \nabla U| < \epsilon
\]
on $X$. Indeed, it is sufficient to prove this inequality in one-dimensional case, choosing a point $x_0$ where the inequality fails and the one-dimensional direction $e$ along which 
\[
\left| (\nabla V - \nabla U) \cdot e\right| = \left| \frac{\partial V}{\partial e} - \frac{\partial U}{\partial e} \right| \ge \epsilon.
\] 
In the one-dimensional case, we have without loss of generality 
\[
U(x_0) - V(x_0) \ge 0, U'(x_0) - V'(x_0) \ge \epsilon.
\]
For the second derivatives, we know that $0 \le U''(x),V''(x) \le 1$ almost everywhere, and therefore for the difference of functions and $x > x_0$ there holds the estimate:
\[
U(x) - V(x) \ge \epsilon (x-x_0) - \frac{1}{2} (x - x_0)^2.
\]
Putting $x = x_0 + \epsilon$ we have $U(x) - V(x) \ge 1/2\epsilon^2$, a contradiction.

Therefore for any positive pancakeness parameter
\begin{equation}
\label{equation:choiceofdelta}
\delta < \frac{\epsilon^2}{4(R+\epsilon)}
\end{equation}
we have the estimate 
\[
|T_i(x) - S(x)| \le \epsilon
\]
on the set $X$. The map $S$ is $1$-Lipschitz, it takes $\nu_t(B)$ to $\nu_t(0)\cap A$ and therefore $T_i$ takes $\nu_{t-\epsilon}(B) \cap X$ to $\nu_t(0)$. In view of \eqref{equation:linearbound} we establish
\begin{multline}
\label{equation:pancakedelta}
\frac{\gamma( \nu_t(c(P_i)) \cap P_i )}{\gamma( P_i )} = \frac{\kappa_i( \nu_t(0))}{\kappa_i(\mathbb R^n)}\ge \frac{\gamma(\nu_{t-\epsilon}(B))}{\gamma(\mathbb R^n)} - \epsilon \ge\\
\ge  \frac{\int_{\nu_{t-\epsilon}(0)} e^{-a_{n-k+1}x_{n-k+1}^2 - \dots - a_n x_n^2}\; dx_{n-k+1}\dots dx_n}{\int_{\mathbb R^k} e^{-a_{n-k+1}x_{n-k+1}^2 - \dots - a_n x_n^2}\; dx_{n-k+1}\dots dx_n} - \epsilon.
\end{multline}

The rest of the proof is done like the argument in \cite{grom2003} or \cite{klartag2016}. Now we choose a sequence $\epsilon_N\to +0$, choose appropriate $R_N$ and $\delta_N \to +0$ depending on $\epsilon_N$ and satisfying \eqref{equation:choiceofdelta}. The Borsuk--Ulam argument ensures that the centers of the pancakes on $N$th stage go to a single point $y_N$ under the map $f$. For such $y_N$, the summation over the pancakes gives the bound
\[
\frac{\gamma\left( \nu_t(f^{-1}(y_N)) \right)}{\gamma(\mathbb R^n)} \ge \frac{\int_{\nu_{t-\epsilon_n}(0)} e^{-a_{n-k+1}x_{n-k+1}^2 - \dots - a_n x_n^2}\; dx_{n-k+1}\dots dx_n}{\int_{\mathbb R^k} e^{-a_{n-k+1}x_{n-k+1}^2 - \dots - a_n x_n^2}\; dx_{n-k+1}\dots dx_n} - \epsilon_N.
\]
We may also assume that $y_N\to y$ from compactness considerations and for such $y$ the above bound becomes precisely the required bound in the limit. 

\begin{remark}
The given proof has a serious simplification compared to the original argument, achieved by using the Caffarelli theorem. It would be nice to have a version of the Caffarelli theorem on the sphere that would allow to simplify the proof of Gromov's waist of the sphere theorem (see e.g. \cite{mem2009}). At the moment we are not completely sure about the precise statement in the spherical case. To start with, we need to handle the following problem: Given a spherical convex set $C\subset \mathbb S^n$ and a hemisphere $H\subset \mathbb S^n$, find a $1$-Lipschitz map sending the uniform measure on $H$ to a multiple of the uniform measure on $C$. After that there remains the question of continuous dependence of the map on $H$ and $C$.

Another possibly useful version would be to find a $1$-Lipschitz map $f : \mathbb S^n\to C$ mapping the uniform measure of the sphere to a multiple of the uniform measure on $C$. The candidate for a centerpoint of $C$ will be a point $y\in C$ whose preimage $f^{-1}(y)$ contains two opposite points of the sphere, like in the Borsuk--Ulam theorem. But with such a statement the continuous dependence of the centerpoint on $C$ will also be problematic.
\end{remark}

\section{Existence of waist theorems for some radial measures}
\label{section:radial-non-gaussian}

Possibly the easiest case to consider is when $\mu$ is uniform on the unit sphere in $\mathbb R^n$. There is Gromov's waist of the sphere theorem \cite{grom2003,mem2009}, but our question is not the same, because unlike that theorem the function is defined inside and outside the sphere in the Euclidean space and the distance we use is the Euclidean distance. To put it short, we have the following non-trivial observation:

\begin{theorem}
A precise $t$-neighborhood waist theorem for $\mu$ distributed uniformly on the unit sphere and maps $f : \mathbb R^n\to\mathbb R^k$ is only possible for $n\le 2$ or $n\le k$.
\end{theorem}

\begin{proof}
The case $n\le k$ is trivial, the origin is the fiber we need. Consider the case $n=2$ and $k=1$, that is we have a continuous function $f : \mathbb R^2\to\mathbb R$.

Let us find the minimal $T$ such that there exists a fiber $f^{-1}(y)$ whose $T$-neighborhood covers the whole circle $\mathbb S^1$ on which $\mu$ is supported, we evidently have $T\le 1$ be considering the fiber passing through the origin. 

For any $p\in \mathbb S^1$ consider the minimum and the maximum of $f$ in the disk $B_p(T)$, they are continuous functions $m(p,T)$ and $M(p,T)$ of $p$ and $T$. By the choice of $y$ and $T$ we have
\[
\forall p\in \mathbb S^1,\quad m(p,T) \le y \le M(p,T).
\]
The assumption of the minimality of $T$ shows that both equalities must become equalities 
\[
m(p,T) = y = M(q, T),
\]
for some $p,q\in\mathbb S^1$, since otherwise the intersection of all segments $[m(p,T), M(p, T)]$ consisted of a nonzero segment (because we are dealing with continuous functions on a compactum) and we could decrease $T$ keeping this intersection non-empty by the uniform continuity of $m$ and $M$.

Now we have two disks $B_p(T)$ and $B_q(T)$ with the property that $f\ge y$ on the former and $f\le y$ on the latter. For any $0<t<T$ the circle $C = \mathbb S^1(\sqrt{1 - t^2})$ intersects both of them.
	\parbox[b]{0.5\textwidth}{
		\begin{center}
			\includegraphics{fig-gaussian-1}\\
		\f \label{fig:BT1} 
		\end{center}
	}
	\hskip 0.6cm
	\parbox[b]{0.4\textwidth}{
		\begin{center}
			\includegraphics{fig-gaussian-2}\\
		\f \label{fig:BT2}	
		\end{center}
	}

Then by the intermediate value theorem there exist at least two points $a,b\in C$ such that $f(a) = f(b) = y$ and such that $a$ and $b$ cannot be simultaneously covered by the interior of a $B_{p'}(T)$ for $p'\in\mathbb S^1$, see Figures \ref{fig:BT1} and \ref{fig:BT2}. Therefore the arcs $\mathbb S^1\cap B_a(t)$ and $\mathbb S^1\cap B_b(t)$ are disjoint and they together show that $f^{-1}(y)+t$ intersects $\mathbb S^1$ by a sufficiently big set.

Now we are going to describe the counterexamples for $n\ge 3$ by a modification of \cite[Remark 5.7]{akopyan2017cat}. We will have two essentially different cases.

Case $k=1$: We choose $p\in\mathbb S^{n-1}$ and start from the function $g(x) = |x - p|$. This function has the $t$-neighborhood property for $t=1$ satisfied on the level $G = \{g = 1\}$ only. Now we modify $g$ near the intersection $G\cap \mathbb S^{n-1}$ so that the new fiber $F = \{f = 1\}$ becomes orthogonal to $\mathbb S^{n-1}$ in their intersection and $F$ still remains the only fiber satisfying the $1$-neighborhood property, the latter is guaranteed by choosing $f\le g$. In fact, this can be made in two-dimensions and extended to higher dimensions by rotation invariance about the $0p$ axis, see Figure \ref{fig:counter1}.

Now we test small $t>0$ against the constructed $f$ and its fiber $F$. The orthogonality condition provides that for $n=2$ the asymptotics of $(F + t)\cap \mathbb S^1$ is correct. For $n\ge 3$ the intersection $I = F\cap\mathbb S^{n-1}$ is an $(n-2)$-dimensional sphere of radius strictly less than $1$ and from the orthogonality it follows that the asymptotics in $t\to +0$ of the surface area of $(F+t)\cap \mathbb S^{n-1}$ is the same as of $(I+t)\cap \mathbb S^{n-1}$, which is insufficient to match $(\mathbb R^{n-1} + t)\cap \mathbb S^{n-1}$, just because $I$ has smaller $(n-2)$ volume than $\mathbb R^{n-1}\cap \mathbb S^{n-1}$.

	\parbox[b]{0.5\textwidth}{
		\begin{center}
			\includegraphics{fig-gaussian-3}\\
		\f \label{fig:counter1} 
		\end{center}
	}
	\hskip 0.6cm
	\parbox[b]{0.4\textwidth}{
		\begin{center}
			\includegraphics{fig-gaussian-4}\\
		\f \label{fig:counter2}	
		\end{center}
	}

Case $2\le k < n$: Let us build a counterexample as follows. The components $f_1,\ldots,f_{k-1}$ will be just the linear coordinates. Considering the value $t=1$ we readily see that the fiber presumably satisfying the $t$-neighborhood waist theorem must lie in 
\[
\mathbb R^{n-k+1} = \{f_1 = \dots = f_{k-1} = 0\},
\]
otherwise the $t$-neighborhood would not reach some of the points with $f_i=\pm 1$.

Moreover, for any definition of the last coordinate $f_k :\mathbb R^n\to\mathbb R^k$ the fiber $F=f^{-1}(y)$ we can presumably take must pass through the origin. Indeed, it already has to lie in $\mathbb R^{n-k+1}$ and in order to reach a point $p$ with $f_1=\pm 1$ with $t=1$ it must pass through the origin since the origin is the only point of $\mathbb R^{n-k+1}$ at distance at most $1$ from $p$. 

Now we pass to $\mathbb R^{n-k+1}$ noting that $n-k+1\ge 2$. It remains to build a function $f_k : \mathbb R^{n-k+1}\to \mathbb R$ whose fiber passing through the origin is very close to a radius segment $[0,p]$, $|p|=1$, see Figure \ref{fig:counter2}. Then the fiber $F$ that we have to choose assuming the $t$-neighborhood waist theorem will be close to a half of $\mathbb R^{n-k}$ instead of the full $\mathbb R^{n-k}$ producing losses in the intersection $(F+t)\cap\mathbb S^{n-1}$ for $0<t<1$.
\end{proof}

In fact the last counterexample extends to arbitrary radially symmetric and compactly supported measure.

\begin{theorem}
\label{theorem:counterexample-k-ge-2}
Let $\mu$ be a radially symmetric compactly supported measure in $\mathbb R^n$ different from the delta-measure at the origin. Then there is no $t$-neighborhood waist theorem for $\mu$ and maps $f : \mathbb R^n\to\mathbb R^k$ when $2\le k < n$.
\end{theorem}

\begin{proof}
By scaling we assume that the support of $\mu$ is precisely the unit ball. Then the last construction in the previous proof works, considering $1$-neighborhoods we need to choose the fiber passing through the origin and arbitrarily close to a half of $\mathbb R^{n-k}$. But this fiber does not work for intermediary $0<t<1$.
\end{proof}

So far we only have positive examples in the plane, but it is also possible to make a counterexample. 

\begin{example}
Take $\mu$ so that half of it is the delta-measure at the origin and the other half is uniformly distributed on the unit sphere. Let the function $f : \mathbb R^2\to\mathbb R$ be just $f(x) = |x|$. Assuming the $t$-neighborhood waist property for this measure and considering $t\to +0$ we see that $y$ has only two possibilities, $y = 0$ and $y=1$.

In both cases considering $t\to 1-0$ exhibits a discontinuous jump of $\mu(f^{-1}(y)+t)$ to its maximal value, which contradicts the inequality
\[
\mu(f^{-1}(y)+t) \ge \mu(\mathbb R + t),
\]
since the right hand part has no jump.
\end{example}

There remain low-dimensional questions, which we still cannot answer:

\begin{question}
Does the uniform measure on the two-dimensional Euclidean disk have the $t$-neighborhood waist property? The same question for the uniform measure on the three-dimensional ball.
\end{question}

Note that the counterexample in \cite[Remark 5.7]{akopyan2017cat} starts working from dimension $4$.

\section{Waist for radial measures and odd maps}
\label{section:odd}

The waist theorems estimate the measure of $f^{-1}(y) + t$, but it is not clear which choice of $y$ is good in every particular situation. There is an easy case when we can choose a particular $y$, the case of odd maps. In \cite{ahk2016} it was noted that for a continuous odd map $f : \mathbb S^n \to \mathbb R^k$ the fiber $f^{-1}(0)$ intersects every $k$-dimensional subsphere $S\subset \mathbb S^n$ at least twice; this is sufficient to invoke Crofton's formula and conclude that the $(n-k)$-dimensional volume (in a certain sense) of the fiber is sufficiently large. But in order to estimate the $n$-dimensional volume of the neighborhood $f^{-1}(0)+t$ we need a trickier argument:

\begin{theorem}
\label{theorem:neighborhood-odd}
Let $f : \mathbb S^n \to \mathbb R^k$ be an odd continuous map and let $t>0$. Then
\[
\vol (f^{-1}(0) + t) \ge \vol (\mathbb S^{n-k} + t).
\]
\end{theorem}

This theorem is an improved version of a particular case of \cite[Theorem 5.1.2]{palic2018}, here we provide a simpler proof of it.

\begin{proof}
The proof of Gromov and Memarian \cite{grom2003,mem2009} mostly works in this case; the thing that needs an adjustment is the Borsuk--Ulam-type theorem for partitions of the sphere with a binary tree hierarchy of cuts. 

If we want a partition into $N=2^\ell$ parts $P_1,\ldots, P_N$, use hyperplanes with the choice of the normal from an $m$-dimensional sphere ($m=k+1$) each time, then the configuration space $\mathcal M$ of the binary partitions with the choice of the normals has dimension $m(N-1)$. This is sufficient to satisfy the constraints
\[
f(c(P_1)) = \dots = f(c(P_N)),\quad \vol P_1 = \dots = \vol P_N.
\]
If we count the constraints then we see that we have precisely $m(N-1)$ constraints. The free group action on $\mathcal M$ in this Borsuk--Ulam-type result is the group of symmetries of the graded binary tree, which is $\Sigma_N^{(2)}$, the $2$-Sylow subgroup of the permutation group.

Now, for centrally symmetric partitions, we have a choice of the normal at the root of the tree, and after this we only have to choose the normals at one of the two subtrees below the root. Thus the configuration space $\mathcal M$ gets reduced to the product of $N/2$ $m$-dimensional spheres, giving a configuration space $\mathcal M_0$ of dimension $mN/2$. Some of the constraints are also satisfied automatically, we only need to solve the following equations:
\[
f(c(P_1)) = \dots = f(c(P_{N/2})) = 0, \quad \vol P_1 - \vol\mathbb S^n/N = \dots = \vol P_{N/2} - \vol\mathbb S^n/N = 0.
\]
The number of constraints written in this way is also $mN/2$. In fact, we have written some redundant volume constraints, but writing the constraints this way shows that we can interpret them as $m$ sets of $N/2$ equations with the same action of the symmetry group inside each set. The flip at the root, for example, will change the sign of all the constraints. The symmetry group this time is $G = \mathbb Z/2 \times \Sigma_{N/2}$, the first factor flips the root, the second is responsible of its subtrees. 

The Borsuk--Ulam-type statement for the modified problem is established by the \emph{test map scheme}, see also \cite{klartag2016}. We consider a linear projection $\pi: \mathbb S^m \to \mathbb R^m$ and replace, for each $i$, the pair 
\[
\left(f(c(P_i)), \vol P_i - \vol\mathbb S^n/n\right)
\] 
with the projection $y_i = \pi(\nu^0_i + s \nu^1_i + \dots + s^\ell \nu^\ell_i)$, where $s$ is a small number and $\nu_i^j$ is the normal that participates in the building of the part $P_i$ on level $j$ of the tree. It is easy to see that such a collection of $(y_1, \ldots, y_N)\in (\mathbb R^m)^N$ keeps the symmetries of $\left(f(c(P_i)), \vol P_i - \vol\mathbb S^n/n\right)_{i=1}^N$ under the action of $\Sigma_N^{(2)}$ and its subgroup $G$.

It is only possible to satisfy the constrains $y_1 = \dots = y_{N/2} = 0$ over $\mathcal M_0$ by choosing $\nu_i^j$ from the pairs of points with $\pi(\nu_i^j) = 0$. Similar to the original problem, this describes a unique non-degenerate orbit of zeroes of the test $G$-equivariant map $\mathcal M_0\to \mathbb R^{nN/2}$. From the parity considerations, for any other $G$-equivariant map $\mathcal M_0\to \mathbb R^{mN/2}$ maps something to zero. This establishes the needed Borsuk--Ulam-type result for mapping the centers of the pancakes into zero. The rest of the proof proceeds as in \cite{grom2003,mem2009}.
\end{proof}

It turns out that for radially symmetric measures in $\mathbb R^n$ and odd maps we have a very general result:

\begin{theorem}
\label{theorem:odd-with-density}
Let $\mu$ be a radially symmetric measure in $\mathbb R^n$ and $f : \mathbb R^n \to \mathbb R^k$ be an odd continuous map. Then
\[
\mu (f^{-1}(0)+t) \ge \mu (\mathbb R^{n-k} + t).
\]
\end{theorem}

\begin{proof}
Put $Z=f^{-1}(0)$, it suffices to consider measures with a radially symmetric density $\rho$, since the general case can be obtained by weak approximation of any measure by measures with density.

Consider the intersection of $Z+t$ and a sphere $\mathbb S^{n-1}_r$ of radius $r$. We will estimate this slice from below by the $t$-neighborhood (in the Euclidean metric) of $Z\cap \mathbb S^{n-1}_s$ with $s=\sqrt{r^2-t^2}$. In terms of the spherical geometry, we consider $Z\cap \mathbb S^{n-1}_s$ in $\mathbb S^{n-1}_s$, consider its $\alpha$-neighborhood in spherical geometry with $\alpha=\arcsin t/r$ and then inflate this neighborhood $r/s$ times to put it onto $\mathbb S^{n-1}_r$, this will be precisely the $t$-neighborhood of $Z\cap \mathbb S^{n-1}_s$ in the Euclidean metric intersected with $\mathbb S^{n-1}_r$.

From Theorem~\ref{theorem:neighborhood-odd} we have an estimate for the $\alpha$-neighborhood of $Z\cap \mathbb S^{n-1}_s$. Then we multiply it by $r/s$, multiply by $\rho(r)$, and integrate over $r$ to have an estimate from below for $\int_{Z+t} \rho(x)\; dx$. There is no need to write down the explicit formulas since for the case $Z=\mathbb R^{n-k}$ we always have a strict equality in all steps of this estimate, which we just put in the right hand side of the total estimate.
\end{proof}

Another result was established in~\cite[Theorem 5.2]{akopyan2017cat} using Gromov's version of the Borsuk--Ulam theorem for the images of the $(1-t)$-scaled centers of the pancakes: Suppose $K\subset\mathbb R^n$ is a convex body, $\mu$ is a finite log-concave measure supported in $K$, and $f : K\to Y$ is a continuous map to a $(n-k)$-manifold $Y$. Then for any $t \in [0, 1]$ there exists $y\in Y$ such that $\mu (f^{-1}(y) + tK) \ge t^{n-k} \mu K$. 

Using the modifications in the Borsuk--Ulam-type argument from the proof of Theorem~\ref{theorem:neighborhood-odd}, we readily obtain its version for centrally symmetric body and measure and an odd map (compare also with \cite[Theorem 5.7]{klartag2016}):

\begin{theorem}
\label{theorem:waist-norm-odd}
Suppose $K\subset\mathbb R^n$ is a centrally symmetric convex body, $\mu$ is a finite centrally-symmetric log-concave measure supported in $K$, and $f : K\to \mathbb R^{n-k}$ is an odd continuous map. Then for any $t \in [0, 1]$
\[
\mu (f^{-1}(0) + tK) \ge t^{n-k} \mu K.
\]
\end{theorem}

\section{Neighborhoods of critical submanifolds}
\label{section:manifold-neighborhood}

Let us develop the ideas of Section \ref{section:odd} about $t$-neighborhoods of given submanifolds. Bo'az Klartag in \cite{klartag2017} established another result of this kind, where a neighborhood volume estimate 
\[
\gamma(f^{-1}(0)+t)\ge \gamma(\mathbb C^{n-k} + t)
\]
for a holomorphic map $f : \mathbb C^n\to\mathbb C^k$ with $f(0)=0$. We try to understand such phenomena in view of the fact that the variety in that case $f^{-1}(0)$ is a critical point of the volume functional under local perturbations.

\subsection{Neighborhoods of submanifolds in the Euclidean space}

Consider a $k$-dimensional smooth submanifold $X\subset\mathbb R^n$; note that we have interchanged $k$ and $n-k$ compared to the statement of the waist theorem, but this will be convenient in our argument. Let $X$ be \emph{properly embedded}, which effectively means that $X$ is a closed submanifold without boundary. let $P_X : \mathbb R^n \to X$ be the metric projection, and let $\mu$ be a measure in $\mathbb R^n$. We assume $\mu$ has a log-concave smooth density $\rho$.

The map $P_X$ is defined almost everywhere and is not necessarily continuous. Still, it induces a fiber-wise decomposition of $\mu$
\begin{equation}
\label{eq:proj-decomposition}
\int_{\mathbb R^n} f d\mu = \int_{X} \left( \int_{P_X^{-1}(x)} f d\mu_x \right) d\vol_k(x)
\end{equation}
where $\mu_x$ is supported in $P_X^{-1}(x)$ for every $x\in X$, so its support has dimension $n-k$.

First let us note that every $\mu_x$ is log-concave. Indeed, the decomposition \eqref{eq:proj-decomposition} can be approximated by the following finite decompositions. Take a sufficiently dense discrete subset $S\subset X$ and build its Voronoi diagram, assigning to every $s\in S$ the convex set
\[
V_s = \{ x\in\mathbb R^n : \forall s'\in X\ |x-s| \le |x-s'|\}.
\]
When taking the limit over more and more dense $S$ the measures $\mu|_{V_s}$ approach $\mu_x$ if $s\to x$. This is quite informal, but can be made precise similar to the argument with passing to infinite partitions in~\cite{grom2003,mem2009}. The Voronoi diagram of $S$ is build by choosing the closest point in the set $S$ to a given point $x'$; in the limit we obtain the presentation of the whole measure $\mu$ as its disintegration into $\mu_x$, that is is we want to integrate a function over $\mu$, we first integrate it over every $\mu_x$ and then integrate over $x\in X$.

What is seen from the Voronoi diagram picture, is that since every restriction $\mu|_{V_s}$ is log-concave, the limit measures $\mu_x$ are also log-concave. If $\mu$ is more log-concave than a Gaussian density $e^{-A|x|^2}$ then $\mu_x$ is also more log-concave than the same density, let us call this situation \emph{strongly log-concave}. 

Now, if we want to estimate $\mu (X+t)$, where $X+t$ is the $t$-neighborhood of $X$, following~\cite{grom2003,mem2009}, we have to estimate $\mu_x B_x(t)$ for every $t$, if we know that 
\begin{equation}
\label{eq:pancake-estimate}
\mu_x B_x(t) \ge C_t \mu_x \mathbb R^n
\end{equation}
from some strong log-concavity assumption, then we integrate to obtain
\[
\mu (X+t) \ge C_t \mu \mathbb R^n.
\]

In order to have a working estimate \eqref{eq:pancake-estimate}, it is preferable to have the situation where $x$ is a point of maximum density of $\mu_x$, apart from the strong log-concavity of $\mu$. Here we consider the density $\rho_x$ of $\mu_x$ in its $(n-k)$-dimensional convex support, which coincides with the normal $(n-k)$-dimensional subspace to $X$ at $x$ locally. More precisely, in a tubular neighborhood of $X$ the density $\rho_x$ is smooth as a function of $y\in P_X^{-1}(x)$ and is also smooth as a function of $x$. 

Another way to consider $\rho_x$ is to say that this is the Jacobian of the exponential map from the normal bundle of $X$ to $\mathbb R^n$, which establishes a diffeomorphism of an open neighborhood $U_X$ of $X$ in its normal bundle with almost all of $\mathbb R^n$, the rest of $\mathbb R^n$ is called the \emph{cut locus}. This is seen from the representation of $\mu_x$ as the disintegration of $\mu$ under the metric projection onto $X$. Therefore the value $\rho_x$ may be considered as the density of the pull-back of the volume in $\mathbb R^n$ to $U_X$.

The condition that $\rho_x$ has a maximum of density at $x$ in the direction orthogonal to $X$ is purely local, in view of its log-concavity, and must have a local expression. In case of the constant density of $\mu$ this is definitely related to the trace of the second fundamental form of $X\subset\mathbb R^n$. Arguing geometrically, when $\mu$ has constant density near $x$, try to deform $X$ along a vector field $v$ supported in a neighborhood of $x$ in $X$ and orthogonal to $X$, we may note that the derivative of $\rho_x$ at points $x'\in\supp v$ in the direction of $v$, averaged over $x'$, vanishes if and only if the first variation of $\vol_k X$ in the direction $v$ vanishes. 

It seems plausible that for the case of not necessarily uniform $\mu$ with density $\rho$, the condition on the central point of $\mu_x$ being $x$ seems to mean vanishing of the first variation of the $\rho$-weighted $k$-dimensional Riemannian volume of $X$. Then a sufficient property of $X$ in order to have a good estimate for $\mu (X+t)$ is that $X$ is a critical point of the $\rho$-weighted Riemannian $k$-volume, say \emph{$\rho$-critical} for short. Let us state what the above argument proofs:

\begin{proposition}
Let a $k$-dimensional properly embedded submanifold $X\subset\mathbb R^n$ be $\rho$-critical for a Gaussian density $\gamma$ centered at the origin with density $\rho$. Then for all $t>0$
\[
\gamma(X+t) \ge \gamma(\mathbb R^k+t).
\]
\end{proposition}

Returning to Klartag's theorem \cite{klartag2017} we observe that a zero set of holomorphic functions $X$ is volume-critical (see Section~\ref{section:algebraic}) but it is not necessarily $\rho$-critical for the Gaussian density. Therefore our observation is insufficient to reprove Klartag's theorem, but we can prove something useful in the sphere $\mathbb S^n$ with its intrinsic Riemannian structure and the Riemannian volume.

\subsection{Neighborhoods in the sphere and the complex projective space}
\label{section:algebraic}

A similar to the above argument works in the sphere $\mathbb S^n$ with its uniform Riemannian volume as $\mu$. In this case instead of log-concavity we need another notion, expressing the fact that the pancake measure $\mu_x$ can be approximated by the uniform measure restricted to convex subsets of $\mathbb S^n$, that is the notion of $\sin^k$-concave measures, introduced in \cite{mem2009}. From the results of \cite{mem2009} we only have to know that once a measure $\mu_x$ is $(n-k)$-dimensionally supported $\sin^k$-concave and is centered at $x$ in terms of the maximum density, there is an estimate
\[
\mu_x(B_x(t)) \ge c_{n,k}(t) \mu_x \mathbb S^n,
\]
where $c_{n,k} = \mu (\mathbb S^k + t) / \mu \mathbb S^n$ for a standardly embedded $\mathbb S^k\subset\mathbb S^n$. 

For a volume-critical $k$-dimensional submanifold $X\subset\mathbb S^n$ and its metric projection $P_X:\mathbb S^n\to X$ we again have a decomposition of the spherical Riemannian volume
\begin{equation}
\label{eq:proj-decomposition-sph}
\int_{\mathbb S^n} f d\vol = \int_{X} \left( \int_{P_X^{-1}(x)} f d\mu_x \right) d\vol_k(x),
\end{equation}
with $\sin^k$-concave measures $\mu_x$ (as shown by approximating this decomposition with a Voronoi partition). The ``center at $x$'' assumption is satisfied for $\mu_x$ by the variational argument, since otherwise the perpendicular to $X$ derivative of the density of $\mu_x$ would be non-zero and it would be possible to variate $X$ with a linear order change of its volume. Hence we have the estimate and integrating the estimate we obtain:

\begin{proposition}
\label{proposition:sphere-critical}
A volume-critical smooth properly embedded submanifold $X\subset\mathbb S^n$ of dimension $k$ satisfies, for $t>0$,
\[
\vol (X + t) \ge \vol (\mathbb S^k + t),
\]
where $\mathbb S^k\subset\mathbb S^n$ is standardly embedded.
\end{proposition}

We can also push forward the spherical observation to the complex projective space:

\begin{theorem}
\label{theorem:proj-analytic}
Let $X\subset\mathbb CP^n$ be an algebraic submanifold with $\dim_{\mathbb C} X = k$. If $\mathbb CP^n$ is considered with its standard Fubini--Study metric then, for $t>0$,
\[
\vol (X + t) \ge \vol (\mathbb CP^k + t),
\]
where $\mathbb CP^k\subset\mathbb CP^n$ is standardly embedded.
\end{theorem}

\begin{proof}
Note that \cite[Section I]{harvey-lawson1982} (also noted previously in \cite{wirtinger1936}, \cite{derham1957}, \cite[\S 4]{federer1965}, \cite{berger1970}) that $X$ is critical in $\mathbb CP^n$ because for any $2k$-dimensional $X'\subset \mathbb CP^n$ we have the \emph{calibration inequality}
\[
\vol_{2k} X' \ge \int_{X'} \frac{\omega^k}{k!},
\]
where $\omega$ is the Fubini--Study symplectic form of $\mathbb CP^n$. And at the same time for a complex subspace $X$ we have 
\[
\vol_{2k} X = \int_X \frac{\omega^k}{k!}.
\]
Since the form $\omega^k$ is closed, the right hand side of the estimates does not change under the deformations of $X$ to $X'$ and then $X$ is not only volume-critical, but also volume-minimal.

Now consider the Hopf map $H : \mathbb S^{2n+1} \to \mathbb CP^n$ and $Y=H^{-1}(X)$. From~\cite[Theorem 2]{hsiang-lawson1971} we know that $Y$ is also volume-critical. 

From  Proposition~\ref{proposition:sphere-critical} we have a lower bound for the volume of the $t$-neighborhood of $Y$, with equality holding for the case $X = \mathbb CP^k$. The Hopf map is a quotient map of Riemannian manifolds with all fibers of length $2\pi$, hence for the $t$-neighborhood of $X$ we will have the same estimate divided by $2\pi$. And again, the resulting estimate must be attained for $X=\mathbb CP^k$, thus completing the proof.
\end{proof}

The above observations allow to establish a particular case of the result of \cite{klartag2017} (for homogeneous maps, generalized to non-Gaussian radial measures):

\begin{theorem}
\label{theorem:radial-symmetric-homo-holo}
Let $Z$ be the zero set of a homogeneous holomorphic map $f : \mathbb C^n\to\mathbb C^k$; let $\mu$ be a radially symmetric measure on $\mathbb C^n$. Then for any $t>0$
\[
\mu(Z + t) \ge \mu(\mathbb C^{n-k} + t).
\]
\end{theorem}
\begin{proof}
By generically perturbing the coordinates of the map $f$ we may assume that every intersection of $Z$ with a sphere $\mathbb S^{2n-1}_s$ of radius $ы$ centered at the origin is a smooth submanifold of the sphere, all such intersections for different $s$ are similar to each other from homogeneity. By the argument from the proof of Theorem \ref{theorem:proj-analytic} every such intersection $Z\cap \mathbb S^{2n-1}_s$ is volume-critical and therefore has an appropriate lower bound on the volume of its $t$-neighborhood for every $t$. The rest of proof is the same as the proof of Theorem \ref{theorem:odd-with-density}.
\end{proof}

The following opposite estimate (an upper bound on the volume of the neighborhood) for algebraic manifolds was prompted to us by Alexander Esterov in private communication:

\begin{theorem}
\label{theorem:degree-algebraic}
Let $X\subset\mathbb CP^n$ be an algebraic submanifold with $\dim_{\mathbb C} X = k$ and degree $d$. If $\mathbb CP^n$ is considered with its standard Fubini--Study metric then, for $t>0$,
\[
\vol (X + t) \le d \vol (\mathbb CP^k + t),
\]
where $\mathbb CP^k\subset\mathbb CP^n$ is standardly embedded.
\end{theorem}

\begin{proof}
We again lift everything to $\mathbb S^{2n+1}$ and make estimates in every pancake. Note that in every pancake the measure $\mu_x$ is $\sin^{2k+1}$-concave and has maximum density $\rho_x$ in the point $x\in X'$. Its $\sin^{2k+1}$-concavity property means that for the ratio $\frac{\mu_x B_x(t)}{\rho_x(x)}$ is bounded from above by the similar ratio for the test case $X=\mathbb CP^k$, $X'=\mathbb S^{2k+1}$, see~\cite[Lemma 4.6]{mem2009} for example.

The integral of $\rho_x(x)$ over $X'$ is just the $(2k+1)$-volume of $X'$, this follows from the Minkowski volume formula. So integrating the estimate
\[
\mu_x B_x(t) \le \rho_x(t) \frac{\vol (\mathbb S^{2k+1} + t) }{\vol_{2k+1} \mathbb S^{2k+1}}
\] 
we obtain
\[
\vol (X' + t) \le \frac{\vol_{2k+1} X' \cdot \vol (\mathbb S^{2k+1} + t) }{\vol_{2k+1} \mathbb S^{2k+1}}.
\]

The volume $\vol_{2k+1} X' = 2\pi \vol_{2k} X$, by Crofton's formula or the calibration equality, is the volume of $\mathbb S^{2k+1}$ multiplied by the degree of $X$, so we have eventually
\[
\vol (X' + t) \le d \vol (\mathbb S^{2k+1} + t),
\]
which is equivalent to the required estimate.
\end{proof}

\begin{remark}
As was communicated to us by Sergei Ivanov, this result can be obtained by analyzing the Riccati equation for the second fundamental form of the hypersurface $H_t = \partial (X'+t)$,
\[
\mathrm{II}_t' = - (\mathrm{II})^2 - R,
\]
with the quadratic form $R$ obtained by plugging the normal vector of $H_t$ into the Riemann curvature form in the last and the first position. This equation shows that the trace of $\mathrm{II}$ in case of $X'$ decreases quicker than in case of $S^{2k+1}$, and so does the logarithm of the volume of the neighborhood, while for small $t$ both volumes have the same asymptotic behavior. This works for smooth $H_t$, but possible non-smoothness of the boundary after some $t$ may only improve the estimate.
\end{remark}

\begin{question}
Is it true that, for any closed $n$-dimensional Riemannian manifold $M$ with sectional curvature bounded from below by $1$, any volume-critical smooth closed $k$-dimensional $X\subset M$, and any $t>0$, we have
\[
\frac{\vol (X+t)}{\vol M} \ge \frac{\vol (\mathbb S^k + t)}{\vol \mathbb S^n}?
\]
\end{question}

The pancake approach does not work in this case, but it seems plausible that the answer follows from the investigation of the evolution of the volume and the second fundamental form of the boundary of $(X+t)$ when $t$ varies from $0$ to a certain value.

There is another question about real projective spaces:

\begin{question}
Does this result generalize to the estimate 
\[
\vol (X+t) \ge \vol(\mathbb RP^k + t)
\]
for $k$-dimensional submanifolds $X\subset\mathbb RP^n$ homologous to $\mathbb RP^k\subset\mathbb RP^n$?
\end{question}

\section{Appendix: Explanation of the Caffarelli theorem}
\label{section:caffarelli}

For the reader's convenience we give a more detailed argument from \cite{koles2011} explaining the Caffarelli theorem. 

\begin{theorem}[Caffarelli]
Assume a measure $\mu_Q$ has a smooth density $e^{-Q}$ on the whole $\mathbb R^n$, while another measure $\mu_P$ has a smooth density $e^{-P}$ on an open convex set, we assume $P$ convex. Assume also that for every $x\in \mathbb R^n$, $y$ in the domain of $\mu_P$, and a nonzero vector $v$ we have
\[
D_v^2 P(y) - D_v^2 Q(x) \ge 0,
\]
where $D^2_v$ denotes the second derivative in the direction of $v$ and also assume that $D_v^2P(y)$ is positive definite at every $y$ as function in $v$. Then the monotone transportation taking $\mu_Q$ to $\mu_P$ is $1$-Lipschitz.
\end{theorem}

\begin{proof}
Under the smoothness and positivity assumptions, the potential $U:\mathbb R^n\to \mathbb R$ of the transportation map from $\mu_Q$ to $\mu_P$ satisfies the equation:
\[
\det D^2 U(x) = e^{P(DU(x))-Q(x)},
\]
where $DU$ is the derivative of $U$ at a point $x$ that gives the transportation map $x\mapsto DU(x)$, and $D^2 U(x)$ is the Hessian quadratic form at a point $x$. If we are interested in the first and the second derivatives in a given direction $v$ we write $D_v U$ or $D_v^2 U$.

Taking the logarithm we obtain
\[
\ln\det D^2 U(x) = P(DU(x)) - Q(x).
\]
Applying additional translations (that do not change the assumptions of the theorem) assume we consider the situation at the origin and the image of the origin under the transportation is again the origin, thus $DU(0) = 0$. In this case the second order of the right hand side in $v$ near the origin is the quadratic form 
\[
D_{\Delta_0 (v)}^2 P (0) - D_v^2 Q(0),
\]
where $D^2 P$ and $D^2 Q$ are regarded as quadratic forms, whose values are taken at the vectors $\Delta_0 (v)$ and $v$ respectively, where $\Delta_0$ is the derivative of the transportation map, that is $D^2 U(0)$ regarded as a linear operator. If we take $v$ to be an eigenvector of $\Delta_0$ and assume its eigenvalue $\lambda$ is greater than $1$, then we see that the value of this quadratic form on $v$ becomes
\[
\lambda^2 D_v^2 P(0) - D_v^2 Q(0) = (\lambda^2 - 1) D^2_v P(0) + D_v^2 P(0) - D_v^2 Q(0),
\]
which is positive by the assumption of the theorem. We are going to show that there is a contradiction, thus showing that $D^2 U$ cannot have eigenvalues greater than $1$ and thus the transportation map is $1$-Lipschitz.

In order to have a contradiction we need to show that the quadratic term of $\ln\det D^2 U(v)$ in its expansion in $v\to 0$ is non-positive, this will be established under a certain choice of a point $x$ (the one we translate to origin) and a vector $v$ in which we expand this quantity. Let again $\Delta_0 = D^2 U(0)$, viewed as a symmetric positive definite matrix and let us check how the expression $D^2 U(x)$ changes when we change $x$, putting $x = tv$ with some fixed nonzero vector $v$ and variable $t\in\mathbb R$. We are interested in the first and the second order, so we assume
\[
D^2 U(tv) = \Delta_0 + \Delta_1 t + \Delta_2 t^2 + o(t^2),
\]
where we consider
\[
\Delta_1 = D^2 D_v U(0),\quad \Delta_2 = \frac{1}{2} D^2 D^2_v U (0)
\]
as symmetric matrices, and $\Delta_0$ can be assumed positive definite in the considered smooth transportation case. In order to simplify the formulas put $A = \Delta_0^{-1/2} \Delta_1 \Delta_0^{-1/2}$, $B = \Delta_0^{-1/2} \Delta_2 \Delta_0^{-1/2}$ and write
\[
\det (\Delta_0 + \Delta_1 t + \Delta_2 t^2 + o(t^2)) = \det \Delta_0\cdot\det (I + At + Bt^2 + o(t^2)), 
\]
where $I$ is the unit matrix. A simple calculation produces the expansion of the logarithm of this matrix expression
\[
\ln\det (\Delta_0 + \Delta_1 t + \Delta_2 t^2 + o(t^2)) = \ln\det \Delta_0 +  \tr A t + \left(\tr B + 
\tr \wedge^2 A - 1/2 (\tr A)^2 \right) t^2 + o(t^2),
\]
where $\tr \wedge^2 A$ is the second invariant polynomial of $A$. If $A$ is diagonalized with eigenvalues $s_1,\ldots, s_n$ then we have
\[
\tr \wedge^2 A - 1/2 (\tr A)^2 = \sum_{i < j} s_i s_j - 1/2 \left( \sum_i s_i \right)^2 = -1/2 \sum_i s_i^2 \le 0. 
\]
Hence we obtain a useful inequality for logarithms of determinants of matrices
\[
\ln\det (\Delta_0 + \Delta_1 t + \Delta_2 t^2 + o(t^2)) \le \ln\det \Delta_0 +  \tr A t + \tr B t^2 + o(t^2)
\]

Now assume we choose the unit vector $v$ and the point $x$ for which the second directional derivative $f(x) = D^2_v U(x)$ is maximal, it can be assumed that the maximum is attained in the situation when the support of $\mu_P$ is compact, in this case $f(x)$ will tend to zero as $|x|\to +\infty$. Other cases are reduced to this by going to the limit with the usage of the stability of the transportation map. So we have $x$ and $v$ producing a maximal $D^2_v U(x)$, as before, we translate this optimal $x$ and its transportation image $D U(x)$ to the origin to simplify the calculations. 

At this point the matrix $\Delta_2 = 1/2 D^2 f = 1/2 D^2 D_v^2 U$ is negative semidefinite from the maximality assumption; hence $B= \Delta_0^{-1/2} \Delta_2 \Delta_0^{-1/2}$ is also negative semidefinite. The maximality assumption on the vector $v$ means that the quadratic form $D_v^2 U (0)$ attains its maximum in $v$; hence $v$ has to be the eigenvector of the corresponding linear operator $\Delta_0$ (the same matrix as $D^2U(0)$) with the maximal possible eigenvalue $\lambda$, which we assumed to be greater than $1$ somewhere and hence greater than $1$ in the situation where it is maximal. In this situation the second derivative of the expression
\[
\ln\det D^2 U (tv)
\]
equals to the trace of a negative semidefinite matrix $B = \Delta_0^{-1/2} \Delta_2 \Delta_0^{-1/2}$ and is non-positive, while the second derivative of 
\[
P(DU (tv)) - Q(tv) = D_{\Delta_0 (tv)}^2 P (0) - D_{tv}^2 Q(0) + o(t^2) = t^2 \left( \lambda^2 D_{v}^2 P (0) - D_{v}^2 Q(0) \right) + o(t^2)
\]
is positive at $t=0$ by the assumption of the theorem. This is a contradiction.
\end{proof}

\section{Appendix: Explanation of the pancake decomposition}
\label{section:pancakes}

In \cite{klartag2016} the method of producing pancakes for Gromov's waist theorem was different compared to the original work \cite{grom2003}. Some of the readers told us that because of this they have an impression that the method to produce the pancake decomposition in \cite{grom2003} was incorrect. That is why we have decided to give more explanations here about the argument from \cite{grom2003}.

Let us assume we are proving Theorem \ref{theroem:pancakes}, although we will sometimes speak about the extension of the argument to its spherical version, where the Euclidean ball of radius $R$ is replaced by the sphere $\mathbb S^n$ and the binary partition is made by hyperplanes through the center of the sphere.

Let us choose a sequence of uniformly distributed linear subspaces $\{L_i\}$ in $\mathbb R^n$ of dimension $n-k-1$ each. Note that we only consider the essential case $n>k$, hence the dimension is at least $0$. In fact we do not need a uniform distribution in the Grassmannian $G_{n-k-1}(\mathbb R^n)$, we will be quite satisfied if the sequence $L_i$ visits any open subset of the Grassmannian infinitely many times.

\subsection{General version of the equipartition argument}
After that we build a binary decomposition of $\mathbb R^n$ into $N=2^I$ parts, on $i$th stage of the decomposition we use hyperplane cuts parallel to $L_i$. Every individual cut (including the cut at infinity) is parameterized by a sphere $S^{k+1}$, the total hierarchy of cuts is parameterized by $(S^{k+1})^{N-1}$. Similar to what is happening for the standard ham sandwich theorem, the generalized Borsuk--Ulam type theorem (from \cite{grom2003}) for $\mathfrak S^{(2)}_N$-equivariant (the $2$-Sylow subgroup of the permutation group) maps 
\[
\left( S^{k+1} \right)^{N-1} \to \left(\mathbb R^{k+1} \right)^{N}/\Delta(\mathbb R^{k+1}),\quad\text{where}\quad \Delta(x) = ( \underbrace{x,\ldots,x}_N )
\]
allows us to find a binary partition with equal measures of parts and equal images $F(P_i)$. The proof of this generalization of the Borsuk--Ulam theorem is essentially a parity counting argument, establishing that generically the number of solution $\mathfrak S^{(2)}_N$-orbits is odd, as explained in \cite{klartag2016}, for example.

For empty or degenerate parts the values $F(P_i)$ are undefined and we need to deal with it in order to apply the Borsuk--Ulam-type theorem. Let $Z$ be the closed subset of $(S^{k+1})^{N-1}$ corresponding to the partitions with equal measures of all parts. Evidently, the values $F(P_i)$ are well-defined on $Z$ and in a neighborhood of it. After that we may modify the functions $F(P_i)$ so that they remain the same over $Z$ and get extended to the whole $(S^{k+1})^{N-1}$ continuously, to achieve this, it is sufficient to multiply them by a continuous function supported in the neighborhood of $Z$. Then the Borsuk--Ulam type theorem is applied to the values $\mu(P_1),\ldots,\mu(P_N)$, and the $k$-dimensional vectors $F(P_1),\ldots, F(P_N)$, now continuously defined over the whole configurations space. Once we obtain a solution from the Borsuk--Ulam-type theorem, the obtained equality $\mu(P_1)=\dots=\mu(P_N)$ guarantees we are on $Z$, and hence on the set where $F(P_i)$ are originally defined.

\subsection{Modification of the equipartition argument for the spherical waist theorem}

The pancake decomposition argument is also used in Gromov's waist of the sphere theorem \cite{grom2003,mem2009}, and in its generalization \cite{karvol2013} for maps to smooth manifolds $f : \mathbb S^n\to M^k$. In this case we use the binary partition of the sphere $\mathbb S^n$ by cuts through the origin, we again choose a uniformly distributed sequence of linear subspaces $L_i\subseteq \mathbb R^{n+1}$ of dimension $n-k-1$ each (assuming the nontrivial case $n>k$). The possible cuts on stage $i$ are made by hyperplanes through $L_i$, and in every node of the binary tree the possible cuts are again parametrized by spheres $S^{k+1}$, the unit spheres of the orthogonal complements of $L_i$. 

The corresponding version of the Borsuk--Ulam theorem in this case is about $\mathfrak S^{(2)}_N$-equivariant maps (not stated explicitly in \cite{karvol2013}, but proven there)
\[
\left( S^{k+1} \right)^{N-1} \to \left(M^k\times \mathbb R \right)^{N}/\Delta(M^k\times \mathbb R),\quad\text{where}\quad \Delta(x) = ( \underbrace{x,\ldots,x}_N ),
\]
which for $M^k = \mathbb R^k$ is the same as above, the factor $\mathbb R$ corresponds to the requirement to equipartition the volume.

In this case we need more care to continuously extend the map $F(\cdot)$ from its original domain, since the target space $M^k$ need not be contractible. Using the fact that $F = f\circ c$, it is sufficient to continuously extend the center map $c(P_i)$ assigning a ``center point'' to a part in the sphere, and then compose the extension of $c$ with $f$. Moreover, since we are going to plug $F$ into the Borsuk--Ulam-type theorem, we may consider a part $P_i$ as a member of the binary partition and let $c(P_i)$ depend continuously on the position of $P_i$ in the partition, minding that the collection of $c(P_i)$ must keep some equivariance under the $\mathfrak S^{(2)}_N$-action. Then we may impose the restriction that the extended center map $c(P_i)$ will also depend on the first partition stage, and the extended $c(P_i)$ has to go to the open hemisphere of the first sphere partition stage to which a possibly empty or degenerate $P_i$ is assigned, this restriction is valid wherever $c(P_i)$ is defined originally. This way the map extension problem is to extend a continuous section of a fiber bundle with contractible fibers, such an extension problem always has a solution. 

What \emph{seems important to us} is that the center point selection procedure for spherical convex sets in \cite{grom2003,mem2009} still remains somewhat complicated and we do not see a Caffarelli-type simplification for the spherical case at the moment.

\subsection{Simplified version of the equipartition argument}
In fact, the application if the Borsuk--Ulam-type theorem in the proof of Theorem~\ref{theroem:pancakes} can be simplified, although this simplification seems to be not suitable in the spherical case. 

Consider \emph{only} binary decompositions with equal measures of parts. It means that on $i$th stage we cut all parts into equal halves by hyperplanes parallel to $L_i$. Every such cut is parameterized by unit vectors $u\perp L_i$, just because after we choose the direction of a hyperplane we in fact find a unique hyperplane that cuts a given part in two parts of equal measure. Since every such $u$ is chosen from a copy of $S^k$, the total hierarchy of equipartition cuts is parameterized by $(S^{k})^{N-1}$. The generalized Borsuk--Ulam type theorem for $\mathfrak S^{(2)}_N$-equivariant maps 
\[
\left( S^{k} \right)^{N-1} \to \left(\mathbb R^{k} \right)^{N}/\Delta(\mathbb R^{k}),\quad\text{where}\quad \Delta(x) = ( \underbrace{x,\ldots,x}_N )
\]
is then applied to the map that sends a binary partition to the collection of values
\[
F(P_1), \ldots, F(P_N)
\]
and gives a configuration with all the $F(P_i)$ equal.

\subsection{Proof that the parts are pancakes}

Having a binary partition into parts of equal measure and with coinciding $F(P_i)$, we need to show that for any given $\delta$ we can choose sufficiently large $N$ so that almost all the parts are pancakes.

The contrary to the needed pancake property is that a part $P_i$ contains a $(k+1)$-dimensional disc $D$ of radius $\delta$, and this happens for arbitrarily large $N$. This can be seen, following \cite{klartag2016}, from the fact that every convex body in $\mathbb R^n$ can be approximated by its John ellipsoid up to scaling by $n$, and if an ellipsoid is not $\delta/n$-close to the affine subspace spanned by its $k$ principal exes (which would imply $P_i$ is $\delta$-close to the same $k$-dimensional affine subspace) then the ellipsoid does contain a $(k+1)$-dimensional disk $D$ of radius $\delta/n$. Thus we assume the contrary and denote $\delta/n$ by $\delta$ for brevity.

An elementary observation shows that if a convex body $K\subset\mathbb R^n$ is cut into equal halves by a hyperplane $h$ then $h$ divides the corresponding directional width of $K$ in ratio at least $1 - 2^{-n} : 2^{-n}$. Since we have fixed the dimension $n$, this is just some small constant. Considering our measure instead of the volume, whose density is in the range $[m, M]$ we will have the same estimate for the width ratio with a rough constant $c_\mu = \frac{m}{M}(1 - 2^{-n})$. 

Now we track the origin of the part that contains the disk $D$, let the affine full of $D$ be $A$. During the production of this part, by the assumption, we have made arbitrarily large number of cuts almost parallel to the $(n-k-1)$-dimensional orthogonal complement of $A$, just because many $L_i$ were close to this complement. Those cuts produced the convex parts of the big ball, 
\[
Q_1\supset Q_2 \supset\dots\supset Q_I = P_i,
\] 
each $Q_{i+1}$ produced from $Q_i$ by an equipartition of its $\mu$-measure. The orthogonal projection $\pi_A(Q_i)$ was evidently inclusion-decreasing. Moreover, for cuts with $L_i$ close to $A^\perp$ there was a direction $u$ in $A$ such that the width $w_u( \pi_A(Q_{i+1}) )$ was bounded from above
\[
w_u( \pi_A(Q_{i+1}) ) \le (1 - c_\mu/2) w_u( \pi_A(Q_{i}) ).
\]
From this it follows that, assuming sufficiently large number of such width-decreasing cuts we get a contradiction with the inclusion $\pi_A(Q_I)\supseteq D$. One way to argue is that we may also assume that the number of cuts with approximately same $u$ was also large and resulted in the decrease of $w_u$ from $R$ to a number smaller than $\delta$. The other way is to note that the $(k+1)$-dimensional volume of $\pi_A(Q_i)$ under cuts with $L_i$ close to $A^\perp$ also had a guaranteed decrease
\[
\vol_{k+1} \pi_A(Q_{i+1})  \le \left(1 - (c_\mu/2)^{k+1}\right) \vol_{k+1} \pi_A(Q_{i}),
\]
which gives a decrease from $\frac{\pi^{(k+1)/2}}{((k+1)/2)!} R^{k+1}$ to $\vol_{k+1} D = \frac{\pi^{(k+1)/2}}{((k+1)/2)!} \delta^{k+1}$ in a finite number of steps, and gives a contradiction so sufficiently large $N$.

If we are interested in the spherical version, then the local situation is essentially the same up to small curvature. Having a $(k+1)$-dimensional disk $D$ of radius $\delta$ (for small $\delta$ it is very close to its Euclidean analogue) in a single part will contradict the fact that many of the cuts were almost passing through the $(n-k-1)$-dimensional orthogonal complement in $\mathbb R^{n+1}$ to the $(k+2)$-dimensional cone over $D$, that is many cuts were essentially perpendicular to $D$ and had to decrease its width in some direction.

\bibliography{../Bib/karasev,../Bib/akopyan}

\begin{thebibliography}{10}

\bibitem{ahk2016}
A.~Akopyan, A.~Hubard, and R.~Karasev.
\newblock Lower and upper bounds for the waists of different spaces.
\newblock {\em Topological Methods in Nonlinear Analysis}, 53(2):457--490,
  2019.
\newblock \href{https://arxiv.org/abs/1612.06926}{arXiv:1612.06926}.

\bibitem{ak2016ball}
A.~Akopyan and R.~Karasev.
\newblock A tight estimate for the waist of the ball.
\newblock {\em Bulletin of the London Mathematical Society}, 49(4):690--693,
  2017.
\newblock \href{http://arxiv.org/abs/1608.06279}{arXiv:1608.06279}.

\bibitem{akopyan2017cat}
A.~Akopyan and R.~Karasev.
\newblock Waist of balls in hyperbolic and spherical spaces.
\newblock {\em International Mathematics Research Notices}, 2018.
\newblock \href{https://arxiv.org/abs/1702.07513}{arXiv:1702.07513}.

\bibitem{berger1970}
M.~Berger.
\newblock Quelques probl\`emes de g\'eom\'etrie {R}iemannienne ou deux
  variations sur les espaces sym\'etriques compacts de rang un.
\newblock {\em Enseignement Math.}, 16:73--96, 1970.

\bibitem{brenier1991}
Y.~Brenier.
\newblock Polar factorization and monotone rearrangement of vector-values
  functions.
\newblock {\em Communications on Pure and Applied Mathematics},
  XLIV(10):375--417, 1991.

\bibitem{caffarelli2000}
L.~A. Caffarelli.
\newblock Monotonicity properties of optimal transportation and the {FKG} and
  related inequalities.
\newblock {\em Communications in Mathematical Physics}, 214(3):547--563, 2000.

\bibitem{derham1957}
G.~de~Rham.
\newblock {\em On the Area of Complex Manifolds. Notes for the Seminar on
  Several Complex Variables}.
\newblock Institute for Advanced Study, Princeton, 1957--1958.

\bibitem{federer1965}
H.~Federer.
\newblock Some theorems on integral currents.
\newblock {\em Transactions of the American Mathematical Society}, 117:43--67,
  1965.

\bibitem{grom2003}
M.~Gromov.
\newblock Isoperimetry of waists and concentration of maps.
\newblock {\em Geometric and Functional Analysis}, 13:178--215, 2003.

\bibitem{harvey-lawson1982}
R.~Harvey and H.~B. Lawson~Jr.
\newblock Calibrated geometries.
\newblock {\em Acta Mathematica}, 148(1):47--157, 1982.

\bibitem{hsiang-lawson1971}
W.-Y. Hsiang and H.~B. Lawson~Jr.
\newblock Minimal submanifolds of low cohomogeneity.
\newblock {\em Journal of Differential Geometry}, 5:1--38, 1971.

\bibitem{karvol2013}
R.~Karasev and A.~Volovikov.
\newblock Waist of the sphere for maps to manifolds.
\newblock {\em Topology and its Applications}, 160(13):1592--1602, 2013.
\newblock \href{http://arxiv.org/abs/1102.0647}{arXiv:1102.0647}.

\bibitem{klartag2016}
B.~Klartag.
\newblock Convex geometry and waist inequalities.
\newblock {\em Geometric and Functional Analysis}, 27(1):130--164, 2017.
\newblock \href{http://arxiv.org/abs/1608.04121}{arXiv:1608.04121}.

\bibitem{klartag2017}
B.~Klartag.
\newblock Eldan's stochastic localization and tubular neighborhoods of
  complex-analytic sets.
\newblock 2017.
\newblock \href{http://arxiv.org/abs/1702.02315}{arXiv:1702.02315}.

\bibitem{koles2011}
A.~V. Kolesnikov.
\newblock Mass transportation and contractions.
\newblock 2011.
\newblock \href{https://arxiv.org/abs/1103.1479}{arXiv:1103.1479}.

\bibitem{mem2009}
Y.~Memarian.
\newblock On {G}romov's waist of the sphere theorem.
\newblock {\em Journal of Topology and Analysis}, 03(01):7--36, 2011.
\newblock \href{http://arxiv.org/abs/0911.3972}{arXiv:0911.3972}.

\bibitem{palic2018}
N.~Pali\'c.
\newblock {\em Grassmannians, measure partitions, and waists of spheres}.
\newblock PhD thesis, Freie Universit\"at Berlin, 2018.
\newblock
  \href{https://refubium.fu-berlin.de/handle/fub188/23043}{refubium.fu-berlin.de/handle/fub188/23043}.

\bibitem{wirtinger1936}
W.~Wirtinger.
\newblock Eine determinantenidentit\"at und ihre anwendung auf analytische
  gebilde und {H}ermitesche massbestimmung.
\newblock {\em Monatsh. Math. Phys.}, 44:343--365, 1936.

\end{thebibliography}

\bibliographystyle{abbrv}
\end{document}